\documentclass[11pt]{amsart}
\usepackage{amsfonts,amssymb,amsmath,amsthm,cmap}
\usepackage[colorlinks,citecolor=blue,urlcolor=black,linkcolor=black]{hyperref}

\newtheorem{thm}{Theorem}[section]
\newtheorem{cor}[thm]{Corollary}
\newtheorem{prop}[thm]{Proposition}
\newtheorem{claim}{Claim}[thm]
\newtheorem{lemma}[thm]{Lemma}
\newtheorem{fact}[thm]{Fact}

\newtheorem{mainthm}{Theorem}

\theoremstyle{definition}
\newtheorem{defn}[thm]{Definition}

\theoremstyle{remark}
\newtheorem{remark}[thm]{Remark}

\newcommand\s{\subseteq}
\newcommand*\axiomfont[1]{\textsf{\textup{#1}}}
\newcommand\ch{\axiomfont{CH}}
\newcommand\gch{\axiomfont{GCH}}

\renewcommand{\mid}{\mathrel{|}\allowbreak}
\renewcommand{\restriction}{\mathbin\upharpoonright}

\DeclareMathOperator{\acc}{acc}
\DeclareMathOperator{\add}{Add}
\DeclareMathOperator{\dom}{dom}
\DeclareMathOperator{\cov}{cov}

\DeclareMathOperator{\bd}{bd}
\DeclareMathOperator{\im}{Im}
\DeclareMathOperator{\h}{ht}
\DeclareMathOperator{\onto}{{\sf onto}}
\renewcommand{\H}{\operatorname{H}}

\title{Diamond on Kurepa trees}
\author{Ziemowit Kostana}
\address{Institute of Mathematics of the Czech Academy of Sciences, Czech Republic}
\email{ziemek314@gmail.com}

\author{Assaf Rinot}
\address{Department of Mathematics, Bar-Ilan University, Ramat-Gan 52900, Israel.}
\urladdr{http://www.assafrinot.com}

\author{Saharon Shelah}
\address{Einstein Institute of Mathematics, The Hebrew University of Jerusalem, Israel \and Department of Mathematics, Rutgers University, U.S.A.}
\urladdr{http://shelah.logic.at}


\begin{document}
\begin{abstract} We introduce a new weak variation of diamond that is meant to only guess the branches of a Kurepa tree.
We demonstrate that this variation is considerably weaker than diamond by proving it is compatible with Martin's axiom.
We then prove that this principle is nontrivial by showing it may consistently fail.
\end{abstract}
\date{April 3, 2024}
\maketitle

\section{Introduction}
Recall that Jensen's diamond principle $\diamondsuit$ is equivalent to the assertion that there exists a sequence $\langle t_\alpha\mid\alpha<\omega_1\rangle$ such that the following two hold:
\begin{itemize}
\item for every $\alpha<\omega_1$, $t_\alpha$ is a function from $\alpha$ to $2$;
\item for every function $f:\omega_1\rightarrow2$, the set
$G(f):=\{\alpha<\omega_1\mid f\restriction\alpha=t_\alpha\}$ is stationary.
\end{itemize}

Kunen proved that $\diamondsuit$ is no stronger than $\diamondsuit^-$ asserting the existence of a sequence $\langle T_\alpha\mid\alpha<\omega_1\rangle$
such that the following two hold:
\begin{itemize}
\item for every $\alpha<\omega_1$, $T_\alpha$ is a countable family of functions from $\alpha$ to $2$;
\item for every function $f:\omega_1\rightarrow2$, the set
$G^-(f):=\{\alpha<\omega_1\mid f\restriction\alpha\in T_\alpha\}$ is stationary.
\end{itemize}

This paper is motivated by a result of the third author \cite{MR603754} who showed it is consistent that some sequence $\langle T_\alpha\mid\alpha<\omega_1\rangle$ witnesses $\diamondsuit^-$ and yet, 
no transversal $\langle t_\alpha\mid\alpha<\omega_1\rangle\in\prod_{\alpha<\omega_1}T_\alpha$ witnesses $\diamondsuit$.
Our idea here is to deal with a further weakening of $\diamondsuit$ in which we limit the collection of functions $f:\omega_1\rightarrow2$
for which $G(f)$ is required to be stationary.
Specifically, this collection will constitute the branches of a Kurepa tree.
To streamline the matter, it is convenient to focus on well-behaved Kurepa trees such as the following.

\begin{defn}\label{Def11} A \emph{binary Kurepa tree} is a collection $T$ such that:
\begin{itemize}
\item $T\s{}^{<\omega_1}2$;
\item $T$ is downward-closed, i.e., for every $t\in T$ and $\alpha<\dom(t)$, $t\restriction\alpha\in T$;
\item $T$ has countable levels, i.e.,
for every $\alpha<\omega_1$, $T_\alpha:=\{ t\in T\mid\dom(t)=\alpha\}$ is countable;
\item The set $\mathcal B(T):=\{f\in{}^{\omega_1}2\mid\forall\alpha<\omega_1\,(f\restriction\alpha\in T_\alpha)\}$ of cofinal branches through $T$
has size $\ge\aleph_2$.
\end{itemize}
\end{defn}

\begin{defn}[Diamond on Kurepa trees]\label{defdiamond}
Suppose that $T$ is a binary Kurepa tree. The axiom $\diamondsuit(T)$ asserts the existence of a transversal $\langle t_\alpha\mid\alpha <\omega_1\rangle\in\prod_{\alpha<\omega_1}T_\alpha$
such that for every $f\in\mathcal B(T)$, the set $G(f):=\{\alpha<\omega_1\mid f\restriction\alpha=t_\alpha\}$ is stationary.
\end{defn}

It is not hard to see that $\diamondsuit\implies\clubsuit\implies\clubsuit_w\implies\diamondsuit(T)$ for every binary Kurepa tree $T$. It turns out, however, that the latter is considerably weaker than the other diamond-type principles,
and in fact $\diamondsuit(T)$ is compatible with Martin's axiom.
This is a corollary to any of the following two results:
\begin{mainthm}\label{thma} It is consistent that there exists a binary Kurepa tree $T$
such that $\diamondsuit(T)$ holds and cannot be killed by a proper forcing.\footnote{A proper forcing may kill the Kurepa-ness of $T$, but then $\diamondsuit(T)$ will hold trivially.}
\end{mainthm}

\begin{mainthm}\label{thmb} If $T$ and $S$ are two binary Kurepa trees and $|\mathcal B(T)|<|\mathcal B(S)|$, then $\diamondsuit(T)$ holds.
\end{mainthm}
The main result of this paper demonstrates that this very weak variation of diamond is nevertheless nontrivial:
\begin{mainthm}\label{thmc} It is consistent that $\diamondsuit(T)$ fails for some binary Kurepa tree~$T$.
Furthermore, such a tree $T$ can be chosen to be either rigid or homogeneous.
\end{mainthm}
The model witnessing the preceding is obtained by a countable support iteration of proper notions of forcing for adding a branch that evades a potential diamond sequence using models as side conditions.
The said Kurepa tree $T$ will start its life in the ground model as a particular Aronszajn tree obtained from $\diamondsuit^+$.

\subsection{Organization of this paper}
In Section~\ref{sec2}, we provide some preliminaries on trees, justify our focus on Kurepa trees that are binary and compare the principle $\diamondsuit(T)$ with other standard set-theoretic hypotheses.
The proof of Theorem~\ref{thmb} will be found there.

In Section~\ref{sec3}, we study indestructible forms of $\diamondsuit(T)$. The proof of Theorem~\ref{thma} will be found there.

In Section~\ref{sec4}, we present a notion of forcing $\mathbb{Q}(T,\,\vec{t})$ to add a branch through an $\aleph_1$-tree $T$ that evades a given potential diamond sequence $\vec t=\langle t_\alpha\mid\alpha<\omega_1\rangle$.
We give a sufficient condition for $\mathbb{Q}(T,\,\vec{t})$ to be proper, and then turn to iterate it.
The proof of Theorem~\ref{thmc} will be found there.

\section{Warm up}\label{sec2}
\subsection{Abstract, Hausdorff and binary trees}
A \emph{tree} is a partially ordered set $\mathbf T=(T,<_T)$ such that, for every $x\in T$,
the cone $x_\downarrow:=\{y\in T\mid y<_T x\}$ is well-ordered by $<_T$; its order type is denoted by $\h(x)$.
For any ordinal $\alpha$, the $\alpha^{\text{th}}$-level of the tree is the collection $T_\alpha:=\{ x\in T\mid\h(x)=\alpha\}$.
The \emph{height} of the tree is the first ordinal $\alpha$ for which $T_\alpha=\emptyset$.
The tree is \emph{normal} iff for every $t\in T$ and every ordinal $\alpha$ in-between $\h(x)$ and the height of the tree,
there exists some $y\in T_\alpha$ with $x<_Ty$.
The tree is \emph{Hausdorff} iff for every limit ordinal $\alpha$ and all $x,y\in T_\alpha$,
if $x_\downarrow=y_\downarrow$, then $x=y$. In particular, a (nonempty) Hausdorff tree has a unique root.

An \emph{$\aleph_1$-tree} is a tree $\mathbf T$ of height $\omega_1$ all of whose levels are countable.
A \emph{Kurepa tree} (resp.~\emph{Aronszajn tree}) is an $\aleph_1$-tree $\mathbf T$
satisfying that the set $\mathcal B(\mathbf T)$ of all uncountable maximal chains in $\mathbf T$ has size $\ge\aleph_2$ (resp.~is empty).
Definition~\ref{defdiamond} generalizes to abstract (possibly, non-Kurepa) $\aleph_1$-trees by interpreting $f\restriction\alpha$ (for $f\in\mathcal B(\mathbf T)$ and $\alpha<\omega_1$)
as the unique element of $T_\alpha$ that belongs to $f$.
Note, however, that if $\mathbf T$ is an $\aleph_1$-tree with no more than $\aleph_1$-many branches, then $\diamondsuit(\mathbf T)$ easily holds.

Hereafter, whenever we talk about a \emph{binary $\aleph_1$-tree},
we mean a set $T$ satisfying the first three bullets of Definition~\ref{Def11},
and we shall freely identify it with the Hausdorff $\aleph_1$-tree $\mathbf T:=(T,{{\subsetneq}})$.

\begin{lemma}\label{lemma21} For every Hausdorff $\aleph_1$-tree $\mathbf T=(T,<_T)$, there exists a binary $\aleph_1$-tree $S$ that is club-isomorphic to $\mathbf T$,
i.e., for some club $D\s\omega_1$, $(\bigcup_{\alpha\in D}T_\alpha,<_T)$ and $(\bigcup_{\alpha\in D}S_\alpha,{\subsetneq})$ are order-isomorphic.
In particular, if $\mathbf T$ is Kurepa, then $S$ is Kurepa and $\diamondsuit(S)$ iff $\diamondsuit(\mathbf T)$.
\end{lemma}
\begin{proof} Let $\mathbf T=(T,<_T)$ be a given Hausdorff $\aleph_1$-tree.
For every $\alpha<\omega_1$, fix an injection $\varphi_\alpha:T_{\alpha+1}\rightarrow\omega$.
Also fix an injective sequence $\langle r_m\mid m<\omega\rangle$ of functions from $\omega$ to $2$.
We shall define an injection $\psi:T\rightarrow{}^{<\omega_1}2$ satisfying the following two requirements:
\begin{itemize}
\item[(1)] for every $t\in T$, $\dom(\psi(t))=\omega\cdot\h(t)$;
\item[(2)] for all $t',t\in T$, $t'<_T t$ iff $\psi(t')\subsetneq\psi(t)$.
\end{itemize}

The definition of $\psi$ is by recursion on the heights of the nodes in $\mathbf T$.
By Clause~(1) we are obliged to send the unique root of $\mathbf T$ to $\emptyset$ and by Clauses (1) and (2), for every $\alpha\in\acc(\omega_1)$ and every $t\in T_\alpha$,\footnote{For a set of ordinals $A$, we write $\acc(A):=\{\alpha\in A\mid\sup(\alpha)=\alpha>0\}$.}
we are obliged to set $\psi(t):=\bigcup\{\psi(t')\mid t'<_T t\}$.\footnote{The injectivity here follows from Hausdorff-ness.}
Thus, the only freedom we have is at nodes of successor levels.
Here, for every $\alpha<\omega_1$ such that $\psi\restriction T_\alpha$ has already been defined, and for every $t\in T_{\alpha+1}$, let $t^-$ denote the immediate predecessor of $t$, and set
$$\psi(t):=\psi(t^-){}^\smallfrown r_{\varphi_\alpha(t)}.$$

Finally, consider $S:=\{ s\restriction\alpha\mid s\in\im(\psi),~\alpha\le\dom(s)\}$.
It is clear that $S$ is a downward-closed subfamily of ${}^{<\omega_1}2$.
Thus, to see that $S$ is a binary $\aleph_1$-tree, it suffices to prove that all of its levels are countable.
However, by Clauses (1) and (2), more is true, namely, for every $\alpha<\omega_1$,
$$S_\alpha=\{ \psi(t)\restriction\alpha\mid t\in T_\alpha\}.$$

Consider the club $D:=\{\alpha<\omega_1\mid\omega\cdot\alpha=\alpha\}$.
Evidently, $\psi$ witnesses that $(\bigcup_{\alpha\in D}T_\alpha,<_T)$ and $(\bigcup_{\alpha\in D}S_\alpha,{\subsetneq})$ are order-isomorphic.
\end{proof}

Any tree $\mathbf T=(T,<_T)$ can be cofinally-embedded in a Hausdorff tree of the form $(S,{\subsetneq})$ by letting $S$ be the
downward closure of the collection of all $s:\alpha\rightarrow T$
that are order-preserving maps from \emph{a successor} ordinal $(\alpha,{\in})$ onto some downward-closed subset of $(T,<_T)$.
The map that sends each $t\in T$ to the unique $s\in S$ with $\max(\im(s),<_T)=t$
embeds $\mathbf T$ to the successor levels of $S$. In case that $\mathbf T$ is non-Hausdorff, there is no better embedding.
Since non-Hausdorff trees lack genuine limit levels and since the definition of diamond on Kurepa trees has to do with stationary sets,
the study here will be focused on the theory of diamond on Hausdorff Kurepa trees.
By Lemma~\ref{lemma21}, then, we may moreover focus on binary Kurepa trees.

\subsection{Diamonds} Earlier on, we pointed out that $\diamondsuit\implies\clubsuit\implies\clubsuit_w\implies\diamondsuit(T)$
for any binary Kurepa tree $T$. We now give the details of the last implication.
\begin{prop} If $\clubsuit_w$ holds, then $\diamondsuit(T)$ holds for every binary Kurepa tree $T$.
\end{prop}
\begin{proof} Suppose that $\clubsuit_w$ holds. Recalling \cite[p.~61]{Sh:544},
this means that we may fix a sequence $\langle A_\alpha \mid\alpha\in\acc(\omega_1)\rangle$
such that each $A_\alpha$ is a cofinal subset of $\alpha$ of order-type $\omega$,
and, for every uncountable $A\s\omega_1$, the following set is stationary:
$$G(A):=\{\alpha\in\acc(\omega_1)\mid A_\alpha\setminus A\text{ is finite}\}.$$

Now, suppose $T$ is a binary Kurepa tree.
Fix a bijection $\pi:\omega_1\leftrightarrow T$.
For every $\alpha\in\acc(\omega_1)$, if there are $t_\alpha\in T_\alpha$ and a finite $F_\alpha\s A_\alpha$ such that
$$t_\alpha=\bigcup\{\pi(\gamma)\mid\gamma\in A_\alpha\setminus F_\alpha\},$$
then we keep this unique $t_\alpha$; otherwise (including the case $\alpha\notin\acc(\omega_1)$), we let $t_\alpha$ be an arbitrary choice of an element of $T_\alpha$.
To see that $\langle t_\alpha\mid\alpha<\omega_1\rangle$ witnesses $\diamondsuit(T)$,
let $f\in\mathcal B(T)$. Consider the following club:
$$D:=\{\delta\in\acc(\omega_1) \mid\pi[\delta]=\bigcup\nolimits_{\alpha<\delta}T_\alpha\}.$$
Note that for every $\delta\in D$, $\delta\le\pi^{-1}(f\restriction\delta)<\min(D\setminus(\delta+1))$.
In particular, the following set is uncountable:
$$A:=\{ \pi^{-1}(f\restriction\delta)\mid\delta\in D\}.$$
We claim that for every $\alpha\in G(A)$, it is the case that $f\restriction\alpha=t_\alpha$.
Indeed, the set $F_\alpha:=A_\alpha\setminus A$ is finite, and for every $\gamma\in A_\alpha\setminus F_\alpha$,
there exists a unique $\delta\in D$ such that $\delta\le\gamma=\pi^{-1}(f\restriction\delta)<\min(D\setminus(\delta+1))$.
So, since $\sup(A_\alpha\setminus F_\alpha)=\alpha$, it is the case that
$$t_\alpha=\bigcup\{\pi(\gamma)\mid\gamma\in A_\alpha\setminus F_\alpha\}=f\restriction\alpha,$$
as sought.
\end{proof}

It follows from the preceding that $\diamondsuit(T)$ is compatible with $\neg\ch$.
An alternative reasoning goes through the next proposition. Namely, add $\aleph_2$ many Cohen reals to a model of $\diamondsuit^+$.

\begin{prop}\label{prop23} Suppose that $\mathbf T$ is a Kurepa tree such that $\diamondsuit(\mathbf T)$ holds, and that $\mathbb P$ is a notion of forcing.
If $\mathbb P$ is $\sigma$-closed or if $\mathbb P^2$ satisfies the ccc, then $\diamondsuit(\mathbf T)$ remains to hold in $V^{\mathbb P}$.
\end{prop}
\begin{proof} By \cite{MR0277379}, a $\sigma$-closed forcing does not add new branches to $\aleph_1$-trees.
By \cite[Lemma~2.2]{MR3072773}, a forcing notion whose square satisfies the ccc
does not add new branches to $\aleph_1$-trees. In addition, stationary sets are preserved by $\sigma$-closed notions of forcing and by ccc notions of forcing.
So, in both cases, the ground model witness to $\diamondsuit(\mathbf T)$ will survive as a witness in $V^{\mathbb P}$.
\end{proof}

Another simple argument shows that an instance of the parameterized diamond principle from \cite{mhd} is enough.

\begin{prop}\label{parametrizeddiamond} If $\Phi(\omega,=)$ holds, then so does $\diamondsuit(T)$ for every binary Kurepa tree $T$.
\end{prop}
\begin{proof} Suppose that $\Phi(\omega,=)$ holds. This means that for every function $F:{}^{<\omega_1}2\rightarrow\omega$,
there exists a function $g:\omega_1\rightarrow\omega$ such that, for every function $f:\omega_1\rightarrow\omega$,
the following set is stationary:
$$G(f):=\{ \alpha<\omega_1\mid F(f\restriction\alpha)=g(\alpha)\}.$$

Now, suppose $T$ is a binary Kurepa tree.
For every $\alpha<\omega_1$, fix an enumeration $\langle t^n_\alpha \mid n<\omega\rangle$ of $T_\alpha$.
Fix a function $F:{}^{<\omega_1}2\rightarrow\omega$ such that for all $\alpha<\omega_1$ and $t\in T_\alpha$,
$$F(t)=\min\{n<\omega\mid t=t_\alpha^n\},$$
and then let $g:\omega_1\rightarrow \omega$ be the corresponding function given by $\Phi(\omega,{=})$.
A moment's reflection makes it clear that the transversal $\langle t_\alpha^{g(\alpha)} \mid\alpha < \omega_1\rangle$ witnesses $\diamondsuit(T)$.
\end{proof}

Recall that a \emph{Souslin tree} is an Aronszajn tree with no uncountable antichains.
Every Souslin tree contains a Souslin subtree that is normal.
As established in \cite{mhd}, forcing with such trees gives an instance of $\Phi(\omega,=)$. A variant of that argument gives the following.
\begin{prop}\label{prop2} Forcing with a normal Souslin tree adds a $\diamondsuit(T)$-sequence for every ground model binary $\aleph_1$-tree $T$.
\end{prop}
\begin{proof} Working in $V$, suppose that $\mathbf S=(S,{<_S})$ is a normal Souslin tree, and that $T$ is a binary $\aleph_1$-tree.
As $\mathbf S$ is normal and $\mathcal B(\mathbf S)=\emptyset$, for each $s\in S$, we may let $\langle s^n \mid n<\omega\rangle$ be a bijective enumeration of some infinite maximal antichain above $s$.
For every $\alpha<\omega_1$, fix an enumeration $\langle t^n_\alpha \mid n<\omega\rangle$ of $T_\alpha$.
Given a generic $G$ for $\mathbf S^*:=(S,{>_S})$, for every $\alpha<\omega_1$, denote by $s_\alpha$ the unique element of $S_\alpha$ that belongs to $G$,
and then let $g(\alpha)$ denote the unique integer $n<\omega$ such that $(s_\alpha)^n$ belongs to $G$.
We claim that the transversal $\langle t_\alpha^{g(\alpha)} \mid\alpha < \omega_1\rangle$ witnesses $\diamondsuit(T)$.

To see this, back in $V$, fix a condition $s\in S$, a name $\dot{f}$ for a cofinal branch through $T$, and a club $C\s\omega_1$;
we shall find an $\alpha\in C$ and an extension of $s$ forcing that $\dot{f}\restriction\alpha$ coincides with $t_\alpha^{\dot{g}(\alpha)}$.
Note that since $\mathbf S^*$ is ccc, we can indeed restrict our attention to ground model clubs.

Fix a countable $M\prec\H_{\omega_2}$, containing $\{\dot{f},s,\mathbf S^*,T,C\}$.
Write $\alpha:=M\cap\omega_1$, and note that $\alpha\in C$, since $C\in M$.
As $\mathbf S$ is normal, fix $s'\in S_\alpha$ extending $s$.
By a folklore fact, $s'$ is an $\mathbf S^*$-generic branch over $M$, so in particular it determines $\dot{f}\restriction \alpha$, and therefore there is an integer $n$ such that
$$s'\Vdash\dot{f}\restriction \alpha = t_\alpha^n.$$
Set $s'':=(s')^n$, so that $s''\Vdash\dot g(\alpha)=n$. Then $s''\mathrel{>_S}s'\mathrel{>_S}s$, and
$$s''\Vdash\dot{f}\restriction \alpha = t_\alpha^{\dot g(\alpha)},$$
as sought.
\end{proof}
\begin{cor} If $T$ is a normal binary almost-Kurepa Souslin tree, then in some ccc forcing extension, $T$ is a binary Kurepa tree and $\diamondsuit(T)$ holds.
\end{cor}
\begin{proof} By definition, as $T$ is almost-Kurepa, in the forcing extension by $\mathbb P:=(T,{\supseteq})$, $T$ is a Kurepa tree. By Proposition~\ref{prop2}, $\diamondsuit(T)$ holds in $V^{\mathbb P}$.
\end{proof}

\subsection{A weaker concept} Devlin \cite[\S2]{MR523488} showed that if there exists a sequence $\langle t_\alpha\mid\alpha<\omega_1\rangle\in\prod_{\alpha<\omega_1}{}^\alpha2$ such that, 
for every function $f:\omega_1\rightarrow2$, there is an infinite ordinal $\alpha<\omega_1$ with $f\restriction\alpha=t_\alpha$, then $\diamondsuit$ holds.
When combined with Theorem~\ref{thmc}, the next proposition shows that this does not generalizes to our context.
\begin{prop} Every Kurepa tree $\mathbf T$ admits a transversal $\langle t_\alpha\mid\alpha<\omega_1\rangle$
such that $G(f):=\{ \alpha<\omega_1\mid f\restriction\alpha=t_\alpha\}$ is uncountable for every $f\in\mathcal B(\mathbf T)$.
\end{prop}
\begin{proof} Suppose that $\mathbf T=(T,<_T)$ is a Kurepa tree.
Recall that for all $f\in\mathcal B(\mathbf T)$ and $\alpha<\omega_1$, $f\restriction\alpha$ denotes the unique element of $T_\alpha$ that belongs to $f$.
Likewise, for every $t\in T$, and $\alpha\le\h(t)$, we denote by $t\restriction\alpha$ the unique element of $T_\alpha$ that is $<_T$-comparable with $t$.

Now, for every limit ordinal $\beta<\omega_1$, fix a surjection $\phi_\beta:\omega\rightarrow T_{\beta+\omega}$,
and then for every $n<\omega$, let $t_{\beta+n}:=\phi_\beta(n)\restriction(\beta+n)$.
To see that $\langle t_\alpha\mid\alpha<\omega_1\rangle$ is as sought, let $f\in\mathcal B(T)$.
For every limit $\beta<\omega_1$, as $f\restriction(\beta+\omega)$ is in $T$, we may find some $n<\omega$ such that $\phi_\beta(n)=f\restriction(\beta+\omega)$.
Consequently, $$f\restriction(\beta+n)=\phi_\beta(n)\restriction(\beta+n)=t_{\beta+n}.$$
It follows that for some stationary $B\s\omega_1$ and some $n'<\omega$, $G(f)\supseteq\{ \beta+n'\mid\beta\in B\}$. In particular, $G(f)$ is uncountable.
\end{proof}
\subsection{Additional ways to get diamond}
Kunen proved that $\diamondsuit$ cannot be introduced by a ccc forcing, whereas Proposition~\ref{prop2} demonstrates that this is not the case with $\diamondsuit(T)$.
Another small ccc forcing that adds a diamond sequence for every ground model Kurepa tree is Cohen's forcing $\add(\omega,1)$.
The core of this fact can be restated combinatorially, as follows.
\begin{prop}\label{prop25} Suppose that $\mathbf T$ is a Kurepa tree.
If $\cov(\mathcal M)>|\mathcal B(\mathbf T)|$,\footnote{That is, if the real line cannot be covered by $|\mathcal B(\mathbf T)|$-many meager sets.}
then $\diamondsuit(\mathbf T)$ holds.
\end{prop}
\begin{proof} For every $\alpha<\omega_1$, fix an enumeration $\langle t^n_\alpha \mid n<\omega\rangle$ of $T_\alpha$.
Fix a partition $\langle S_i\mid i<\omega\rangle$ of $\omega_1$ into stationary sets,
and let $\pi:\omega_1\rightarrow\omega$ be such that $\pi[S_i]=\{i\}$ for every $i<\omega$.
For each $f\in\mathcal B(\mathbf T)$, define a map $r_f:\omega\rightarrow\omega$ via:
$$r_f(i):=\min\{n<\omega\mid\{\alpha\in S_i\mid f\restriction\alpha=t_\alpha^n\}\text{ is stationary}\}.$$

Finally, assuming $\cov(\mathcal M)>|\mathcal B(\mathbf T)|$, by \cite[Lemma~2.4.2]{MR1350295},
we may fix a function $g:\omega\rightarrow\omega$ such that $r_f\cap g\neq\emptyset$ for every $f\in\mathcal B(\mathbf T)$.
Then the transversal $\langle t_\alpha^{g(\pi(\alpha))} \mid\alpha < \omega_1\rangle$ witnesses $\diamondsuit(\mathbf T)$.
\end{proof}
\begin{remark} $\diamondsuit^+$ implies that $\diamondsuit(\mathbf T)$ holds for some Kurepa tree $\mathbf T$ with $\cov(\mathcal M)<|\mathcal B(\mathbf T)|$.
In Corollary~\ref{cor45} below, we get a model in which $\diamondsuit(\mathbf T)$ holds for some Kurepa tree $\mathbf T$ with $\cov(\mathcal M)=|\mathcal B(\mathbf T)|$.
\end{remark}
\begin{remark} A proof similar to that of Proposition~\ref{prop25} shows that $\diamondsuit(\mathbf T)$ holds for every Kurepa tree $\mathbf T$ with $\mathfrak d_{\omega_1}>|\mathcal B(\mathbf T)|$.
\end{remark}
Our next task is proving Theorem~\ref{thmb}. For this, let us recall the following definition.
\begin{defn}A tree $\mathbf T=(T,<_T)$ of height $\omega_1$ is: \begin{itemize}
\item a \emph{weak Kurepa tree} iff $|T|=\aleph_1<|\mathcal B(\mathbf T)|$;
\item \emph{thick} iff $|\mathcal B(\mathbf T)|=2^{\aleph_1}$.
\end{itemize}
\end{defn}

For example, $\ch$ implies that ${({}^{<\omega_1}2,{\subsetneq})}$ is a thick weak Kurepa tree.

\begin{prop}\label{prop28} Suppose that $\mathbf T$ is a Kurepa tree, $\mathbf S$ is a weak Kurepa tree,
and $|\mathcal B(\mathbf T)|<|\mathcal B(\mathbf S)|$. Then $\diamondsuit(\mathbf T)$ holds.
\end{prop}
\begin{proof} By \cite[Corollary~7.18]{paper53}, if $\kappa$ is a regular uncountable cardinal and there is a Kurepa tree with (at least) $\kappa$-many branches,
then $\onto(\{\aleph_1\},\allowbreak J^{\bd}[\kappa],\aleph_0)$ holds.
The same proof works equally well against weak Kurepa trees.
Therefore, $\kappa:=|\mathcal B(\mathbf T)|^+$ is a regular uncountable cardinal such that $\onto(\{\aleph_1\},\allowbreak J^{\bd}[\kappa],\aleph_0)$ holds.
This means that we may fix a map $c:\omega_1\times\kappa\rightarrow\omega$ such that for every $B\in[\kappa]^\kappa$,
there exists $i<\omega_1$ such that $c[\{i\}\times B]=\omega$.
For each $\beta<\kappa$, define $g_\beta:\omega_1\rightarrow\omega$ via $g_\beta(i):=g(i,\beta)$.
It is easy to check that for every $\mathcal R\in[{}^{\omega_1}\omega]^{<\kappa}$,
there exists some $\beta<\kappa$ such that $r\cap g_\beta\neq\emptyset$ for every $r\in\mathcal R$.
From this point on, the proof continues the same way as that of Proposition~\ref{prop25},
where the only change is that we fix a partition $\langle S_i\mid i<\omega_1\rangle$ into $\aleph_1$-many stationary sets
and so each of the $r_f$'s is now a function from $\omega_1$ to $\omega$.
\end{proof}
\begin{cor} $\ch$ implies $\diamondsuit(\mathbf T)$ for every non-thick Kurepa tree $\mathbf T$.\qed
\end{cor}

\section{Indestructible diamonds}\label{sec3} We would like to show that if the universe is close to being constructible, then there exists a Kurepa tree on which $\diamondsuit$ holds. 
The construction of such a tree relies on the concept of a \emph{sealed} Kurepa tree:
\begin{defn}[Hayut-M\"uller, \cite{muellerhayut}] A Kurepa tree $\mathbf T$ is \emph{sealed} if for every notion of forcing $\mathbb P$ that preserves both $\omega_1$ and $\omega_2$,
$\mathbb P$ does not add a new branch through $\mathbf T$.
\end{defn}

\begin{fact}[Po\'or-Shelah, {\cite[\S4]{poorshelah}}]\label{poor-shelah} If $\omega_1=\omega_1^{\mathsf{L}[A]}$ and $\omega_2=\omega_2^{\mathsf{L}[A]}$ for some $A\s\omega_1$, 
then there exists a Kurepa tree $\mathbf T$ such that $\mathcal B(\mathbf T)\s\mathsf{L}[A]$.\footnote{This was generalized by Hayut and M\"uller \cite[Lemma~15]{muellerhayut} to any successor of a regular uncountable cardinal $\kappa$ such that $\kappa^+=(\kappa^+)^{\mathsf{L}}$.}
\end{fact}

\begin{cor}[Giron-Hayut, \cite{giron2023sealed}] If $\omega_1=\omega_1^{\mathsf{L}[A]}$ and $\omega_2=\omega_2^{\mathsf{L}[A]}$ for some $A\s\omega_1$, then there exists a sealed Kurepa tree.\qed
\end{cor}

\begin{cor}\label{cor44}
Suppose that $V=\mathsf{L}[A]$ for some $A\s\omega_1$.
Then there exists a binary Kurepa tree $T$ such that $\diamondsuit(T)$ holds in any forcing extension preserving $\omega_1$, $\omega_2$, and the stationary subsets of $\omega_1$.
\end{cor}
\begin{proof}
Let $\mathbf T$ be a tree as in Fact~\ref{poor-shelah}.
It can be verified that $\mathbf T$ is Hausdorff,
but, regardless, as described right after Lemma~\ref{lemma21},
there is a Hausdorff tree $\mathbf S$ and an order-preserving injection $\pi$ from $\mathbf T$ to $\mathbf S$ such that
$\pi[T_\alpha]=S_{\alpha+1}$ for every $\alpha<\omega_1$.
In particular, $\pi$ induces a bijective correspondence between $\mathcal B(\mathbf T)$ and $\mathcal B(\mathbf S)$,
meaning that we may as well assume that $\mathbf T$ is Hausdorff.
Next, as $A\s\omega_1$, it is the case that $\diamondsuit$ holds in $\mathsf{L}[A]$, so $\diamondsuit(\mathbf T)$ holds as well.
Suppose $\mathbb P$ is a notion of forcing preserving $\omega_1$ and $\omega_2$. The construction of $\mathbf T$ takes place inside $\mathsf{L}[A]$, so
$$V^{\mathbb P} \models \mathcal B(\mathbf T)\s\mathsf{L}[A].$$
As the definition of $\mathsf{L}[A]$ is absolute, all the branches through $\mathbf T$ from $V^{\mathbb P}$ are already in $V$.
Now, if $\mathbb P$ also preserves stationary subsets of $\omega_1$, then the witness to $\diamondsuit(\mathbf T)$ in $V$
will still witness $\diamondsuit(\mathbf T)$ in $V^{\mathbb P}$.
Finally, by running the translation procedure of Lemma~\ref{lemma21},
we obtain a binary Kurepa tree with the same key features as the Hausdorff tree $\mathbf T$.
\end{proof}

\begin{cor}\label{cor45} It is consistent that $2^{\aleph_0}=\aleph_2$, Martin's axiom holds, and $\diamondsuit(T)$ holds for some binary Kurepa tree $T$.
\end{cor}
\begin{proof} Work in $\mathsf{L}$. By Corollary~\ref{cor44},
we may fix a binary Kurepa tree $T$ such that $\diamondsuit(T)$ holds in any forcing extension preserving $\omega_1$, $\omega_2$, and the stationary subsets of $\omega_1$.
Now, let $\mathbb P$ be some ccc notion of forcing such that in $\mathsf L^{\mathbb P}$, Martin's axiom holds and $2^{\aleph_0}=\aleph_2$.
As $\mathbb P$ preserves $\omega_1$, $\omega_2$, and the stationary subsets of $\omega_1$,
$\diamondsuit(T)$ holds in $\mathsf L^{\mathbb P}$.
\end{proof}

A Kurepa tree that is sealed for all proper forcings can also be added by forcing, and in fact such a notion of forcing was already devised by D.~H.~Stewart in his 1966 Master's thesis.
\begin{defn}[Stewart (see \cite{trees})]\label{Stewart} For a cardinal $\kappa>\aleph_1$,
the forcing $\mathbb{S}_\kappa$ consisting of all triples $(T_p,\epsilon_p,b_p)$ such that:
\begin{enumerate}
\item $\epsilon_p\in\acc(\omega_1)$;
\item $T_p$ is a countable downward-closed subfamily of ${}^{\le\epsilon_p}2$ such that $(T_p,{\subsetneq})$ is a normal tree of height $\epsilon_p+1$;
\item $ b_p$ is an injection from a countable subset of $\kappa$ to the top level of $T_p$,
\end{enumerate}
and the ordering is given by $q \le p$ iff
\begin{itemize}
\item $\epsilon_q\ge \epsilon_p$;
\item $T_q\supseteq T_p$ and $T_q\restriction(\epsilon_p+1)=T_p$;
\item $\dom(b_q)\supseteq\dom(b_p)$;
\item for every $\xi\in\dom(b_p)$, $b_q(\xi)\supseteq b_p(\xi)$.
\end{itemize}
\end{defn}

Note that $\mathbb S_\kappa$ is $\sigma$-closed, and that, assuming $\ch$, it has the $\aleph_2$-cc.
The following establishes Theorem~\ref{thma}.

\begin{prop}\label{prop47} Suppose $\kappa>\aleph_1$ is some cardinal.
In the forcing extension by $\mathbb S_\kappa$,
there exists a binary Kurepa tree $T$ such that $\diamondsuit(T)$ holds and no proper forcing adds a new branch through $T$, let alone kills diamond over it.
\end{prop}
\begin{proof} Suppose $G$ is $\mathbb S_\kappa$-generic.
Let $T:=\bigcup\{T_p\mid p\in G\}$ denote the generic binary tree added by $\mathbb{S}_\kappa$.
For each $\xi<\kappa$, let
$$f_\xi:=\{ t\in T\mid\exists p\in G\,(\xi\in\dom(b_p)\wedge t\s b_p(\xi)\}.$$
A density argument shows that $T$ is an $\aleph_1$-tree and that $\langle f_\xi\mid\xi<\kappa\rangle$ is an injective sequence of branches through it.
By a theorem of Baumgartner, any $\sigma$-closed forcing that adds a new subset of $\omega_1$ forces $\diamondsuit$ to hold (see \cite[Theorem~3.1]{sakai}), hence $\diamondsuit(T)$ hold in $V[G]$.

In \cite[p.~10]{trees}, Jech proved that $\{f_\xi\mid\xi< \kappa\}$ enumerates \emph{all} cofinal branches through $T$.
It is a folklore fact that this remains the case in any extension by a proper notion of forcing.
Hints of this fact may be found in \cite[Lemma~8.14]{MR776625} and \cite[Proposition~37]{MR2139743},
but for completeness, we give the details here.
Following \cite[Definition~1]{MR2139743}, we consider the following set:
$$\mathcal S:=\{X\in [\kappa]^{\aleph_0} \mid X\cap\omega_1\in\acc(\omega_1),\; T_{X\cap\omega_1}=\{f_\xi \restriction(X\cap\omega_1)\mid\xi\in X\}\}.$$
\begin{claim} $\mathcal S$ is stationary.
\end{claim}
\begin{proof} Back in $V$, pick a name $\dot{C}$ for a club in $[\kappa]^{\aleph_0}$, and a condition $p$.
Recursively construct a sequence $\langle(p_n,X_n)\mid n<\omega\rangle$ such that $(p_0,X_0):=(p,\omega)$ and, for every $n<\omega$:
\begin{itemize}
\item $p_{n+1}\le p_n$;
\item $\epsilon_{p_{n+1}}>\epsilon_{p_n}$;
\item $\dom(b_{p_{n+1}})\supseteq X_n$;
\item every element of $T_{{p_n}}$ is extended by some element of $\im(b_{p_{n+1}})$;
\item $X_{n+1}\in[\kappa]^{\aleph_0}$ with $X_{n+1}\supseteq X_n\cup\dom(b_{p_n})\cup(\sup(X_n\cap\omega_1)+1)$;
\item $p_{n+1}\Vdash X_{n+1}\in\dot{C}$.
\end{itemize}

Set $X:=\bigcup_{n<\omega}X_n$ and note that $X\cap\omega_1\in\acc(\omega_1)$. Define a condition $q$ by letting:
\begin{enumerate}
\item $\dom(b_q):=\bigcup_{n<\omega}X_n$;
\item for all $n<\omega$ and $\xi\in X_n$, $b_q(\xi):=\bigcup_{n<m<\omega}b_{p_m}(\xi)$;
\item $\epsilon_q:=\sup_{n<\omega}\epsilon_{p_n}$;
\item $T_q:=\bigcup_{n<\omega}T_{p_n}\cup\im(b_{q})$.
\end{enumerate}

It is clear that $q\le p_n$ for all $n<\omega$, and hence $q\Vdash X\in\dot{C}$.
In addition, Clause~(4) implies that $q\Vdash X\in\dot{S}$.
\end{proof}
Now, let $\mathbb{Q}$ be any proper forcing in $V[G]$, and work in $V[G][H]$, where $H$ is $\mathbb Q$-generic.
Towards a contradiction, suppose that $f$ is a branch through $T$ distinct from any of the $f_\xi$'s.
As $\mathbb Q$ is proper, $\mathcal S$ remains stationary, so we may fix an elementary submodel $M\prec\H_\theta$ (for a large enough regular cardinal $\theta$)
containing $\{T,f,\langle f_\xi\mid\xi<\kappa\rangle\}$ such that $X:=M\cap\kappa$ is in $\mathcal S$.
By elementarity, $$M\models \forall \xi <\kappa\,(f \neq f_\xi).$$
Denote $\alpha:=X\cap\omega_1$. As $f$ is a branch through $T$, we get that $f\restriction\alpha\in T_\alpha$.
As $X\in\mathcal S$, we may find a $\xi\in X$ such that $f\restriction \alpha=f_\xi\restriction \alpha$.
But $\xi\in X = M\cap\kappa$, and then, by elementarity,
$$M \models f=f_\xi.$$
This is a contradiction.
\end{proof}

Despite the fact that the generic tree added by $\mathbb{S}_\kappa$ is sealed for proper forcings, we can still add new branches to it without collapsing cardinals or adding reals using the quotient forcing $\mathbb S_{\kappa^+}/\mathbb S_{\kappa}$.
More generally, consider the countable support iteration $(\langle \mathbb P_\xi\mid\xi\le\kappa\rangle,\langle \dot{\mathbb{Q}}_\xi \mid\xi<\kappa\rangle)$, for an uncountable cardinal $\kappa$, where $\mathbb{Q}_0$ is the Jech partial order for adding a Souslin tree $T$, and, for every nonzero $\xi<\kappa$, $$\mathbb{P}_\xi\Vdash\dot{\mathbb Q}_\xi = T.$$
By Proposition~\ref{prop47} and the upcoming proposition, $\mathbb{P}_\kappa$ is proper (even $\sigma$-strategically closed) for any choice of $\kappa$,
and yet the quotient forcings of the form $\mathbb{P}_\mu / \mathbb{P}_\lambda$ are not proper, whenever $\aleph_2 \le \lambda< \mu$.

\begin{prop} $\mathbb{P}_\kappa$ has a dense subset that is isomorphic to $\mathbb{S}_\kappa$.
\end{prop}
\begin{proof} Let $R$ denote the collection of all \emph{rectangular} conditions in $\mathbb P_\kappa$, i.e, the collection of all conditions $q$ for which there exists some $\delta<\omega_1$ such that:
\begin{itemize}
\item $q(0)$ is a tree of height $\delta+1$,
\item for every $\xi\in\kappa\setminus\{0\}$, $q(\xi)$ is either trivial or a node at the $\delta^{\text{th}}$-level of $q(0)$,
\item all maximal nodes of $q(0)$ which are of the form $q(\xi)$, for $\xi\in\kappa \setminus \{0\}$, are pairwise distinct.
\end{itemize}

To see that $R$ is isomorphic to $\mathbb{S}_\kappa$, note that we can map a condition $q\in R$ to a condition $\phi(q)$ in $\mathbb{S}_\kappa$, by keeping the tree coordinate, and using the function $b_{\phi(q)}$ to record all maximal nodes of $q(0)$ that are of the form $q(\xi)$. More precisely, we put:
\begin{itemize}
\item $T_{\phi(q)}(0):=q(0)$,
\item $\epsilon_{\phi(q)}:=\h(q(0))-1$,
\item $b_{\phi(q)}(\xi):=q(\xi),$ whenever $q(\xi)$ belongs to the top level of $q(0)$.
\end{itemize}

This mapping is clearly order-preserving, with the image
$$\{p\in\mathbb{S}_\kappa \mid 0 \notin \dom(b_p)\},$$
which is obviously isomorphic to $\mathbb{S}_\kappa$. Moreover, it is straightforward to define an inverse of $\phi$.

It remains to show that $R$ is dense. Let $p$ be an arbitrary condition in $\mathbb{P}_\kappa$.
Fix a countable $M\prec\H_\theta$, for a sufficiently large regular cardinal $\theta$, containing all relevant objects.
Let $\delta:=\omega_1 \cap M$, and let $G\s P\cap M$ be a $\mathbb{P}_\kappa$-generic filter over $M$, containing $p$. For $\xi\in\kappa \cap M$, let us denote by $G(\xi)$ the projection of $G$ to the coordinate $\xi$.
Therefore $\bigcup G(0)$ is a tree of height $\delta$, with countable levels,
and for each $\xi\in\kappa \cap M \setminus \{0\}$, $G(\xi)$ determines a cofinal branch $f_\xi$ through $G(0)$, given by the formula
$$f_\xi := \bigcup\{ t\in {}^{<\omega_1}2 \mid\exists r\in G\,(r\restriction \xi\Vdash r(\xi) = t)\}.$$

The latter follows from the observation that for any $\epsilon<\delta$, the set
$$D_{\xi,\epsilon}:=\{ r\in\mathbb{P}_\kappa \mid\exists t\in {}^{<\omega_1}2\,(r\restriction \xi\Vdash r(\xi) =t \wedge \dom(t) \ge \epsilon)\},$$
is dense in $\mathbb{P}_\kappa$, and belongs to $M$.

Finally, we define a condition $q$ by the following considerations:
\begin{itemize}
\item $q(0):=G(0)\cup\{f_\xi \mid\xi\in M\cap\kappa\}$,
\item $q(\xi):=f_\xi$ for every $\xi\in M\cap\kappa \setminus \{0\}$,
\item $q(\xi):=\emptyset$ for every $\xi\in\kappa\setminus M$.
\end{itemize}

Then $q$ is a rectangular condition extending $p$, as sought.
\end{proof}

\section{The failure of diamond on a Kurepa tree}\label{sec4}

The main result of this section is a consistent example of a binary Kurepa tree on which diamond fails.
At the end of this section, we shall derive a few additional corollaries.

As a first step, we present a notion of forcing with side conditions for adding a branch through a given normal binary $\aleph_1$-tree $T$.
The poset takes a transversal $\vec t$ of the full binary tree $({}^{<\omega_1}2,{\subsetneq})$ as a second parameter,
and ensures that the generic branch will disagree on a club with this transversal.
\begin{defn}
Let $T$ be a normal binary $\aleph_1$-tree, and let $\vec t\in\prod_{\alpha<\omega_1}{}^\alpha2$.

The forcing $\mathbb{Q}(T,\,\vec{t})$ consists of all triples $p=(x_p,\mathcal M_p,f_p)$ that satisfy all of the following:
\begin{enumerate}
\item $x_p$ is a node in $T$;
\item $\mathcal M_p$ is a finite, $\in$-increasing chain of countable elementary submodels of $\H_{\omega_1}$;
\item $f_p$ is a partial function from $\mathcal M_p$ to $\omega_1$;
\item for every $M\in\mathcal M_p$:
\begin{itemize}
\item $\dom(x_p)\ge M\cap\omega_1$,
\item $x_p \restriction(M\cap\omega_1)\neq\vec t(M\cap\omega_1)$;
\item $f_p\restriction M\in M$.
\end{itemize}
\end{enumerate}
The ordering is defined by letting $q \le p$ iff
\begin{enumerate}
\item $x_q\supseteq x_p$;
\item $\mathcal M_q \supseteq \mathcal M_p$;
\item for every $M\in\dom(f_p)$, $M\in\dom(f_q)$ and $f_q(M)\ge f_p(M)$.
\end{enumerate}
\end{defn}
\begin{remark}\label{rmk53} $\mathbb{Q}(T,\,\vec{t})\s\H_{\omega_1}$, so that $\mathbb{Q}(T,\,\vec{t})\in\H_{\omega_2}$.
\end{remark}
\begin{lemma}\label{generic}
Suppose that $T$ is a normal binary $\aleph_1$-Souslin tree,
and $\vec t\in\prod_{\alpha<\omega_1}{}^\alpha2$.
Let $p\in\mathbb{Q}(T,\,\vec{t})$.
Then for every countable $M^*\prec\H_{\omega_2}$ with $\mathbb{Q}(T,\,\vec{t})\in M^*$
such that $M^*\cap\H_{\omega_1}\in\mathcal M_p$, $p$ is $M^*$-generic.
\end{lemma}
\begin{proof}
Fix a model $M^*$ as above, and let $D\in M^*$ be a dense open subset of $\mathbb{Q}(T,\,\vec{t})$; we need to find an $r\in D\cap M^*$ compatible with $p$.

Fix a large enough $\gamma\in M^*\cap\omega_1$ such that $N\cap\omega_1\s\gamma$ for every $N\in\mathcal M_p\cap M^*$. Define a condition
$$\bar p:=(x_p\restriction\gamma,\mathcal M_p\cap M^*,f_p\restriction M^*),$$
and note that $\bar p\in M^*$.
\begin{claim} The set $D':=\{x_q \mid q\in D,\; q \le \bar p\}$ is dense in $(T,{\supseteq})$ below $x_{\bar p}$.
\end{claim}
\begin{proof}
Let $y$ be any extension of $x_{\bar p}$. Evidently, $(y,\mathcal M_{\bar p},f_{\bar p})$ is a legitimate condition.
Pick $q\in D$ such that $q \le (y,\mathcal M_{\bar p},f_{\bar p})$.
Then $x_q$ is an extension of $y$ that lies in $D'$, as sought.
\end{proof}

Denote $\delta:=M^*\cap\omega_1$.
Since $T$ is a Souslin tree lying in $M^*$, any node in $T_\delta$ is $(T,{\supseteq})$-generic over $M^*$. In particular, $x_p\restriction \delta$ is $(T,{\supseteq})$-generic over $M^*$, and it follows that there is an $x\in D'\cap M^*$ such that
$$x\s x_p\restriction \delta.$$
Now, pick any $q$ witnessing $x\in D'$, so that $x=x_q\s x_p$. By elementarity, we can assume that $q\in M^*$.

Define a condition $r$ by letting $x_r:=x_p$, $\mathcal M_r:=\mathcal M_p\cup\mathcal M_q$, and $f_r$ be such that $\dom(f_r)=\dom(f_p)\cup\dom(f_q)$ and
$$f_r(\alpha):=\max(\{f_p(\alpha)\mid\alpha\in\dom(f_p)\}\cup\{f_q(\alpha)\mid\alpha\in\dom(f_q)\}).$$

It is not hard to see that $r$ is a legitimate condition extending both $p$ and $q$.
\end{proof}
\begin{lemma}\label{properness}
Suppose that $T$ is a normal binary $\aleph_1$-Souslin tree,
and $\vec t\in\prod_{\alpha<\omega_1}{}^\alpha2$. Then:
\begin{enumerate}
\item $\mathbb{Q}(T,\,\vec{t})$ is proper;
\item $\mathbb{Q}(T,\,\vec{t})$ adds a branch through $T$ that evades $\vec{t}$ on a club.
\end{enumerate}
\end{lemma}
\begin{proof}
\begin{enumerate}
\item To verify properness, fix $p\in\mathbb{Q}(T,\,\vec{t})$, and a countable $M^* \prec\H_{\omega_2}$, containing everything relevant including $p$.
Denote $\delta:=M^*\cap\omega_1$.
As $T$ is normal and Souslin, we may be able to extend $x_p$ to some node $x\in T_\delta$ distinct from $\vec t(\delta)$.

The triple $(x,\mathcal M_p \cup\{M^*\cap\H_{\omega_1}\},f_p)$ is a condition in $\mathbb{Q}(T,\,\vec{t})$, and by Lemma~\ref{generic}, it is an $M^*$-generic condition.

\item Let $G$ be a generic filer. It suffices to verify that the following uncountable set is moreover closed:
$$D:=\{M\cap\omega_1 \mid\exists p\in G\,(M\in\mathcal M_p)\},$$
that is, every $\gamma\in\kappa\setminus D$ is not an accumulation point of $D$.

Let $\dot{D}$ be a name for the above set, and suppose $p\Vdash\dot{\gamma}\in\kappa\setminus \dot{D}$. By extending $p$, we can assume that it forces that $\dot{\gamma}=\gamma$, and $\gamma$ lies between two consecutive elements $N_0\in N_1$ of $\mathcal M_p$.
By extending $p$ further, we can assume that $f_p(N_0) \ge \gamma$. Since $\gamma$ is forced to not be the height of any element of $\mathcal M_p$, it follows that $N_0\cap\omega_1 < \gamma$, and
$$p\Vdash\dot{D}\cap(N_0\cap\omega_1,\gamma) =\emptyset,$$
as sought. \qedhere
\end{enumerate}
\end{proof}
\begin{defn} Given a sequence of binary $\aleph_1$-trees $T^0,\ldots,T^{n+1}$,
the tree product $T^0\otimes\cdots\otimes T^n$ is the collection of all $(x_0,\ldots,x_n)\in T^0\times\cdots\times T^n$
such that $\dom(x_0)=\cdots=\dom(x_n)$, and the ordering is such that a node $(x_0,\ldots,x_n)$ is below a node $(y_0,\ldots,y_n)$ iff $x_i\s y_i$ for all $i\le n$.
\end{defn}
In our context, it will be useful to know of the following fact.

\begin{fact}[{\cite[\S5]{paper65}}]\label{diamondplus}
If $\diamondsuit^+$ holds, then there is a binary $\aleph_1$-Aronszajn tree $T$ and a sequence $\vec T=\langle T^\eta\mid\eta<\omega_2\rangle$ of normal binary $\aleph_1$-subtrees of $T$ such that,
for every nonempty $a\in[\omega_2]^{<\omega}$, the tree product $\bigotimes_{\eta\in a}T^\eta$ is an $\aleph_1$-Souslin tree.
\end{fact}

It is well-known that if the product $T\otimes S$ of two normal binary $\aleph_1$-trees is Souslin, then
$$(T,{\supseteq})\Vdash ``S \text{ is Souslin}".$$
Thus, it is natural to expect that furthermore
$$\mathbb{Q}(T,\,\vec{t})\Vdash ``S \text{ is Souslin}".$$
The next lemma shows that this is indeed the case.
\begin{lemma}\label{Souslinpreservation}
Suppose $T$ and $S$ are normal binary $\aleph_1$-trees the product of which is Souslin, and let $\vec t\in\prod_{\alpha<\omega_1}{}^\alpha2$. Then
$$\mathbb{Q}(T,\,\vec{t})\Vdash ``S \text{ is Souslin}".$$
\end{lemma}
\begin{proof}
Let $q\in\mathbb{Q}(T,\,\vec{t})$, and fix a $\mathbb{Q}(T,\,\vec{t})$-name $\dot{A}$ for a maximal antichain in $S$.
Fix a countable $M^* \prec\H_{\omega_2}$ containing $\{\mathbb{Q}(T,\,\vec{t}), q, \dot{A}\}$.
Denote $\delta := M^*\cap\omega_1$, and pick a node ${x}\in T_\delta\setminus\{\vec t(\delta)\}$ extending $x_q$.
We define an extension $q'\le q$ by letting:
$$q':=(x,\mathcal M_q \cup\{M^*\cap\H_{\omega_1}\},f_q).$$
It is clear that $q'$ is indeed a condition and $q' \le q$. We claim that
$$q'\Vdash ``\dot{A}\cap(S\restriction \delta)\text{ is a maximal antichain in }S".$$

To see this, pick a condition $p \le q'$ and a node ${y}\in S_\delta$;
our aim is to find a condition $r\le p$ that forces ${y}$ to extend an element of $\dot{A}$.

Fix a large enough $\gamma<\delta$ such that $N\cap\omega_1\s\gamma$ for every $N\in\mathcal M_p\cap M^*$.
Let $\bar x:=x_p\restriction\gamma$ and $\bar y:={y}\restriction\gamma$.
Define a condition $$\bar p := (\bar x,\mathcal M_p\cap M^*,f_p \restriction M^*).$$

Now, we turn to run a recursion producing a family $$\mathcal{E}=\{(q_\alpha,y_\alpha)\mid\alpha<\theta\},$$ satisfying that for each $\alpha$:
\begin{enumerate}
\item $q_\alpha\in\mathbb{Q}(T,\,\vec{t})$ with $q_\alpha \le\bar p$;
\item $y_\alpha\in S$ with $\bar y\s y_\alpha$;
\item $q_\alpha\Vdash\exists a\in\dot{A}\,(a\s y_\alpha)$;
\item $\dom(x_{q_\alpha})=\dom(y_\alpha)$;
\item for all $\beta<\alpha$, $(x_{q_\beta},y_\beta)\bot(x_{q_\alpha},y_\alpha)$ in $T\otimes S$.
\end{enumerate}
We continue the construction until it is not possible to choose the next element. Because $T\otimes S$ is Souslin,
Requirement~(5) ensures that the construction terminates after countably many steps, so the ordinal $\theta$ ends up being countable.
\begin{claim}\label{claim551} The set $\{(x_{q_\alpha},y_\alpha) \mid\alpha<\theta\}$ is a maximal antichain above $(\bar x,\bar y)$ in $T \otimes S$.
\end{claim}
\begin{proof} Suppose not. Then we may pick $(x',y')\in T \otimes S$ extending $(\bar x,\bar y)$ that is incompatible with $(x_{q_\alpha},y_\alpha)$ for every $\alpha<\theta$.
By possibly extending both $x'$ and $y'$, we may assume that $x'=x_{q'}$ for some condition $q' \le\bar p$ which moreover satisfies:
$$q'\Vdash\exists a\in\dot{A}\,(a\s y').$$
But now we can add $(q',y')$ to $\mathcal{E}$, contradicting its maximality.
\end{proof}

Since $q\in M^*$, by making canonical choices in the recursive construction we may secure that the countable family $\mathcal{E}$ be a subset of $M^*$.
So, by Claim~\ref{claim551}, we may find an $\alpha<\theta$ such that ($x_{q_\alpha}\s x$ and $y_\alpha\s y$). Since $\mathcal{E}\s M^*$, we infer that $q_\alpha\in M^*$.
Now, as in the proof of Lemma~\ref{generic}, we see that there exists a condition $r$ extending both $p$ and $q_\alpha$.
Recalling Requirement~(3) in the choice of $q_\alpha$, we get that
$$r\Vdash\exists a\in\dot{A}\,(a\s {y}),$$
as sought.
\end{proof}

We are now in conditions to prove the core of Theorem~\ref{thmc}.
\begin{thm}\label{thm58} It is consistent that there exists a binary Kurepa tree $T$ such that $\diamondsuit(T)$ fails.
\end{thm}
\begin{proof} We start with a model of $\gch$ and $\diamondsuit^+$.
Using Fact~\ref{diamondplus}, fix a binary $\aleph_1$-tree $T$ and a sequence $\vec T=\langle T^\xi\mid\xi<\omega_2\rangle$
of normal binary $\aleph_1$-subtrees of $T$ such that $\bigotimes_{\xi\in a}T^\xi$ is an $\aleph_1$-Souslin tree
for every nonempty $a\in[\omega_2]^{<\omega}$.\footnote{The tree $T$ given by the fact is moreover Aronszajn, but this feature is not necessary for our application here.}
Note that for all $\xi\neq\eta$, the trees $T^\xi$ and $T^\eta$ have a countable intersection.

Let $\mathcal P$ denote the collection of all pairs $(\mathbb R,\tau)$ such that $\mathbb R$ is a notion of forcing lying in $\H_{\omega_2}$,
and $\tau$ is a nice $\mathbb R$-name for an element of $\prod_{\alpha<\omega_1}T_\alpha$.
Using $\gch$, we may fix a (repetitive) enumeration $\langle(\mathbb R_\xi,\tau_\xi)\mid\xi<\omega_2\rangle$ of $\mathcal P$ in such a way that every pair is listed cofinally often.

Finally, we force with $\mathbb P_{\omega_2}$, where $(\langle \mathbb P_\xi\mid\xi\le\omega_2\rangle,\langle \dot{\mathbb{Q}}_\xi \mid\xi<\omega_2\rangle)$ is the countable support iteration satisfying that for every $\xi<\omega_2$:
\begin{enumerate}
\item[(i)] $\mathbb P_\xi\Vdash ``\dot{\mathbb{Q}}_\xi=\dot{\mathbb{Q}}(T^\xi,\sigma_\xi)"$, and
\item[(ii)] $\sigma_\xi$ is a nice $\mathbb P_\xi$-name for an element of $\prod_{\alpha<\omega_1}T_\alpha$
such that if $\mathbb R_\xi=\mathbb P_\eta$ for some $\eta\le\xi$, then $\sigma_\xi$ is the lift of $\tau_\xi$ (from a $\mathbb P_\eta$-name to a $\mathbb P_\xi$-name).
\end{enumerate}

Recalling Remark~\ref{rmk53}, we infer that $(\mathbb P_\xi,\tau_\xi)\in\mathcal P$ for all $\xi<\omega_2$.

\begin{claim} For every $\xi<\omega_2$, $\mathbb P_\xi$ forces that $\dot{\mathbb{Q}}(T^\xi,\sigma_\xi)$ is proper.
\end{claim}
\begin{proof} By Lemma~\ref{properness}(1), it suffices to prove that for every $\xi<\omega_2$,
$$\mathbb P_\xi\Vdash ``T^\xi\text{ is Souslin}".$$

We shall prove by induction on $\xi<\omega_2$ a stronger claim, namely that for every finite tuple $\xi<\eta_0<\ldots<\eta_n<\omega_2$,
$$\mathbb P_\xi\Vdash ``T^{\xi}\otimes T^{\eta_0}\otimes\cdots \otimes T^{\eta_n}\text{ is Souslin}".$$

There are three cases to consider:
\begin{itemize}
\item $\xi=0$. This is since the sequence $\vec T$ was given by Fact~\ref{diamondplus}.
\item $\xi+1$. Recall that
$$\mathbb P_{\xi+1}\simeq\mathbb P_\xi \ast \dot{\mathbb{Q}}(T^\xi,\sigma_\xi).$$
Let $\xi+1<\eta_0<\ldots<\eta_n<\omega_2$ be a finite tuple.
By the induction hypothesis,
$$\mathbb P_\xi\Vdash ``T^\xi \otimes T^{\xi+1} \otimes T^{\eta_0} \otimes \cdots \otimes T^{\eta_n}\text{ is Souslin}".$$
By invoking Lemma~\ref{Souslinpreservation} in ${V}^{\mathbb P_\xi}$, we infer that
$$\mathbb P_\xi \ast \dot{\mathbb{Q}}(T^\xi, \,\sigma_\xi)\Vdash ``T^{\xi+1} \otimes T^{\eta_0} \otimes \cdots \otimes T^{\eta_n}\text{ is Souslin}",$$
as sought.
\item $\xi\in\acc(\omega_2)$. We apply a result from \cite[\S3]{Sh:403} or \cite{miyamoto} stating that for any $\aleph_1$-tree $S$, the property
``\emph{$S$ is Souslin}'' is preserved at a limit stage of a countable support iteration of proper forcings.\qedhere
\end{itemize}
\end{proof}
A standard name counting argument shows that $\ch$ is preserved in every intermediate stage, so $\mathbb P_{\omega_2}$ satisfies the $\aleph_2$-cc.
In addition, $\mathbb P_{\omega_2}$ is proper, so our forcing preserves all cardinals.
\begin{claim} $\mathbb P_{\omega_2}$ forces that $T$ is Kurepa.
\end{claim}
\begin{proof} For every $\xi<\omega_2$, $\mathbb P_{\omega_2}$ projects to a forcing of the form $\mathbb{Q}(T^\xi, \,\vec t)$,
and then Lemma~\ref{properness}(2) implies that in $V^{\mathbb P_{\omega_2}}$,
there exists a branch $f_\xi$ through $T^\xi$, and in particular, through $T$.
As the elements of $\langle T^\xi\mid\xi<\omega_2\rangle$ have a pairwise countable intersection,
in $V^{\mathbb P_{\omega_2}}$, $\langle f_\xi\mid\xi<\omega_2\rangle$ is injective.
\end{proof}
\begin{claim} $\mathbb P_{\omega_2}$ forces that $\diamondsuit(T)$ fails.
\end{claim}
\begin{proof} Otherwise, as $\prod_{\alpha<\omega_1}T_\alpha$ lies in the ground model,
we may fix a nice $\mathbb P_{\omega_2}$-name $\dot{t}$ for a transversal $\vec t\in\prod_{\alpha<\omega_1}T_\alpha$
that witnesses $\diamondsuit(T)$ in the extension.
As $\mathbb P_{\omega_2}$ has the $\aleph_2$-cc and $\dot{t}$ is a nice name for an $\aleph_1$-sized set,
there is a large enough $\eta<\omega_2$ such that all nontrivial conditions appearing in $\dot{t}$ belong to $\mathbb P_\eta$.
It thus follows that $\vec t$ admits a nice $\mathbb P_\eta$-name, say, $\tau$.
Clearly, $(\mathbb P_\eta,\tau)\in\mathcal P$, so we may find a large enough $\xi\in[\eta,\omega_2)$ such that $(\mathbb R_\xi,\tau_\xi)=(\mathbb P_\eta,\tau)$.
Recalling Clauses (i) and (ii) in the definition of our iteration,
it is the case that $\mathbb P_{\xi+1}\simeq\mathbb P_\xi \ast \dot{\mathbb{Q}}(T^\xi, \sigma_\xi)$
where $\sigma_\xi$ is a $\mathbb P_\xi$-name for $\vec t$.
By Lemma~\ref{properness}(2), then, $\mathbb P_{\xi+1}$ introduces a branch through $T$ that evades $\vec t$ on a club.
So $\vec t$ cannot witness $\diamondsuit(T)$ in $\mathbb P_{\omega_2}$. This is a contradiction.
\end{proof}
This completes the proof.
\end{proof}

\subsection{Ramifications} Proposition~\ref{prop28} may suggest that the validity of $\diamondsuit(T)$ for a Kurepa tree $T$ only depends on the cardinal $|\mathcal B(T)|$.
However, the next corollary shows that this is not the case:
\begin{cor} It is consistent that there exist two binary Kurepa trees $S$ and $T$ such that $|\mathcal B(S)|=|\mathcal B(T)|=2^{\aleph_1}$, $\diamondsuit(S)$ holds, but $\diamondsuit(T)$ fails.
\end{cor}
\begin{proof} Work in $\mathsf L$. By Corollary~\ref{cor44}, we may fix a binary Kurepa tree $S$ such that $\diamondsuit(S)$ holds in any forcing extension preserving $\omega_1$, $\omega_2$, and the stationary subsets of $\omega_1$.
As $\mathsf L$ satisfies $\gch$ and $\diamondsuit^+$,
the proof of Theorem~\ref{thm58} provides us with an $\aleph_2$-cc proper notion of forcing $\mathbb P_{\omega_2}\s\H_{\omega_2}$ that introduces a binary Kurepa tree $T$ on which $\diamondsuit$ fails.
Altogether, $\mathsf L^{\mathbb P_{\omega_2}}$ is a model satisfying the desired configuration.
\end{proof}
\begin{cor} It is consistent that there exists a binary Kurepa tree $T$ such that $\diamondsuit(T)$ fails and $T$ is uniformly homogeneous.\footnote{The definition of a \emph{uniformly homogeneous} tree may be found in \cite[\S4]{paper65}.}
\end{cor}
\begin{proof} The proof is almost identical to that of Theorem~\ref{thm58}. We start with a model of $\gch$ and $\diamondsuit^+$.
Using Fact~\ref{diamondplus}, we fix a binary $\aleph_1$-tree $T$ and a sequence $\langle T^\xi\mid\xi<\omega_2\rangle$
of normal binary $\aleph_1$-subtrees of $T$ such that $\bigotimes_{\xi\in a}T^\xi$ is an $\aleph_1$-Souslin tree
for every nonempty $a\in[\omega_2]^{<\omega}$.
Now, let $T'$ be the collection of all functions $t'\in{}^{<\omega_1}2$ for which there exists some $t\in T$ such that $\dom(t')=\dom(t)$
and $\{ \alpha\in\dom(t)\mid t(\alpha)\neq t'(\alpha)\}$ is finite. Then $T'$ is a uniformly homogeneous $\aleph_1$-tree
having each of the $T^\xi$'s as a normal binary $\aleph_1$-subtree. So, we can continue with the proof of Theorem~\ref{thm58} using $T'$ instead of $T$.
\end{proof}
\begin{cor} It is consistent that there exists a binary Kurepa tree $S$ such that $\diamondsuit(S)$ fails and $S$ is rigid.
\end{cor}
\begin{proof} The proof is quite close to that of Theorem~\ref{thm58}.
We start with a model of $\gch$ and $\diamondsuit^+$.
Instead of using Fact~\ref{diamondplus}, we appeal to \cite[\S5]{paper65} to obtain a downward closed subfamily $T\s{}^{<\omega_1}\omega$
such that $(T,{\subsetneq})$ is an Aronszajn tree, every node in $T$ admits infinitely many immediate successors,
and there exists a sequence $\langle T^\xi\mid\xi<\omega_2\rangle$ of normal downward-closed subtrees of $T$ such that $\bigotimes_{\xi\in a}T^\xi$ is an $\aleph_1$-Souslin tree
for every nonempty $a\in[\omega_2]^{<\omega}$. We force with $\mathbb P_{\omega_2}$, where
$(\langle \mathbb P_\xi\mid\xi\le\omega_2\rangle,\langle \dot{\mathbb{Q}}_\xi \mid\xi<\omega_2\rangle)$ is the countable support iteration satisfying that for every $\xi<\omega_2$,
$\mathbb P_\xi\Vdash ``\dot{\mathbb{Q}}_\xi=\dot{\mathbb{Q}}(T^\xi,\sigma_\xi)"$,
where $\sigma_\xi$ is a nice $\mathbb P_\xi$-name for an element of $\prod_{\alpha<\omega_1}T_\alpha$ obtained from some bookkeeping sequence.
This time, a typical transversal $\vec t\in\prod_{\alpha<\omega_1}T_\alpha$ is an element of $\prod_{\alpha<\omega_1}{}^\alpha\omega$
instead of $\prod_{\alpha<\omega_1}{}^\alpha2$, but everything goes through and we end up in a generic extension in which $T$ is a Kurepa tree on which diamond fails.

Next, let $S$ be the binary $\aleph_1$-tree produced by the proof of Lemma~\ref{lemma21} when fed with the tree $\mathbf T:=(T,{\subsetneq})$.
Since $\mathbf T$ is Kurepa on which diamond fails, $S$ is a binary Kurepa tree and $\diamondsuit(S)$ fails.

Recall that the proof of Lemma~\ref{lemma21} made use of an injective sequence $\langle r_m\mid m<\omega\rangle$ of functions from $\omega$ to $2$.
For our purpose here, we shall moreover assume that for all $m\neq m'$, it is the case that $\Delta(r_m,r_{m'})=\min\{m,m'\}$.\footnote{This can easily be arranged by defining $r_m:\omega\rightarrow2$ via $r_m(n):=1$ iff $n<m$.}

To prove that $S$ is rigid, we must establish the following.
\begin{claim} Suppose that $\pi:S\leftrightarrow S$ is an automorphism of $(S,{\subsetneq})$. Then $\pi$ is the identity map.
\end{claim}
\begin{proof} Suppose not, so that $\pi(s)\neq s$ for some $s\in S$. By the definition of $S$ and as $\pi$ is order-preserving,
it follows that there exists a $t\in T$ such that $\pi(\psi(t))\neq\psi(t)$.
Let $\beta<\omega_1$ be the least for which there exists a $t\in T_\beta$ with $\pi(\psi(t))\neq\psi(t)$.
Clearly, $\beta$ is a successor ordinal, say, $\beta=\alpha+1$.
Fix $t_0\neq t_1$ in $T_{\alpha+1}$ such that $\pi(\psi(t_0))=\psi(t_1)$
and note that the minimality of $\beta$ implies that $t_0\restriction\alpha=t_1\restriction\alpha$, which we hereafter denote by $t$.
Now, since every node in $T$ admits infinitely many immediate successors, we may fix $t_2,t_3\in T_{\alpha+1}$ such that:
\begin{itemize}
\item $t_2,t_3$ are immediate successors of $t$,\footnote{Possibly, $t_2=t_3$.}
\item $\pi(\psi(t_2))=\psi(t_3)$, and
\item $\min\{\varphi_\alpha(t_2),\varphi_\alpha(t_3)\}>\max\{\varphi_\alpha(t_0),\varphi_\alpha(t_1)\}$.
\end{itemize}

Recalling the proof of Lemma~\ref{lemma21}, for every $i<4$, letting $m_i:=\varphi_\alpha(t_i)$, it is the case that
$$\psi(t_i)=\psi(t){}^\smallfrown r_{m_i}.$$

As $m_0<m_2$, $\Delta(r_{m_0},r_{m_2})=m_0$ and
$$\Delta(\psi(t_0),\psi(t_2))={\omega\cdot\alpha}\mathrel{+}m_0.$$
Likewise, $$\Delta(\pi(\psi(t_0)),\pi(\psi(t_2)))=\Delta(\psi(t_1),\psi(t_3))={\omega\cdot\alpha}\mathrel{+}m_1.$$
As $m_0\neq m_1$, we infer that
$$\Delta(\psi(t_0),\psi(t_2))\neq \Delta(\pi(\psi(t_0)),\pi(\psi(t_2))),$$
contradicting the fact that $\pi$ is an automorphism of $S$.
\end{proof}
This completes the proof.
\end{proof}

\section*{Acknowledgments}

The first author was supported by the European Research Council (grant agreement ERC-2018-StG 802756) and by the GA\v{C}R project EXPRO 20-31529X and RVO: 67985840.
The second author was supported by the European Research Council (grant agreement ERC-2018-StG 802756) and by the Israel Science Foundation (grant agreement 203/22).
The third author was supported by the Israel Science Foundation (grant agreement 2320/23) and by the grant ``Independent Theories'' NSF-BSF, (BSF 3013005232).

This is paper no.~66 in the second author's list of publications and paper no.~1252 in the third author's list of publications.

\end{document}